\documentclass[10pt]{amsart}
\usepackage{amssymb,amstext,amsmath,amscd,amsthm,amsfonts,comment,enumerate,latexsym,stmaryrd,multicol,geometry,graphicx,mathrsfs,bm}
\usepackage[usenames]{color}
\usepackage[all]{xy}
\newtheorem{thm}{Theorem}[section]
\newtheorem{lem}[thm]{Lemma}
\newtheorem{prop}[thm]{Proposition}
\newtheorem{cor}[thm]{Corollary}
\theoremstyle{definition}
\newtheorem{dfn}[thm]{Definition}

\newtheorem{rem}[thm]{Remark}

\newtheorem{ex}[thm]{Example}

\theoremstyle{remark}

\newtheorem*{claim*}{Claim}
\newtheorem*{ac}{Acknowledgments}
\newtheorem*{conv}{Convention}

\numberwithin{equation}{thm}

\def\add{\operatorname{\mathsf{add}}}

\def\cm{\mathsf{CM}}

\def\core{\operatorname{\mathsf{core}}}

\def\db{\operatorname{\mathsf{D^b}}}
\def\depth{\operatorname{depth}}
\def\dim{\operatorname{dim}}
\def\ds{\operatorname{\mathsf{D_{sg}}}}

\def\Ext{\operatorname{Ext}}

\def\G{\mathcal{G}}
\def\ge{\geqslant}
\def\H{\mathrm{H}}
\def\hom{\operatorname{Hom}}

\def\id{\operatorname{id}}
\def\inf{\operatorname{inf}}

\def\ker{\operatorname{Ker}}
\def\le{\leqslant}

\def\m{\mathfrak{m}}
\def\max{\operatorname{max}}

\def\mod{\operatorname{\mathsf{mod}}}

\def\p{\mathfrak{p}}
\def\pd{\operatorname{pd}}

\def\res{\operatorname{\mathsf{res}}}

\def\sup{\operatorname{sup}}
\def\syz{\Omega}
\def\T{\mathcal{T}}

\def\thick{\operatorname{\mathsf{thick}}}

\def\X{\mathcal{X}}
\def\Y{\mathcal{Y}}

\def\ZZ{\mathbb{Z}}
\begin{document}
\title{Contravariantly finite resolving subcategories over commutative local rings are trivial ones} 
\author{Yuki Mifune}
\address[Mifune]{Graduate School of Mathematics, Nagoya University, Furocho, Chikusaku, Nagoya 464-8602, Japan}
\email{yuki.mifune.c9@math.nagoya-u.ac.jp}
\author{Gen Tanigawa}
\address[Tanigawa]{Graduate School of Mathematics, Nagoya University, Furocho, Chikusaku, Nagoya 464-8602, Japan}
\email{tanigawa.gen.y4@s.mail.nagoya-u.ac.jp}
\begin{abstract}
We show that a contravariantly finite resolving subcategory over a henselian local ring must be the subcategory of free modules or the whole module category or the subcategory of maximal Cohen--Macaulay modules. This removes the Gorenstein assumption from a theorem of Takahashi. We also show that a resolving subcategory of finite type other than the subcategory of free modules must be the subcategory of maximal Cohen--Macaulay modules. As a consequence, every Cohen--Macaulay local ring of finite CM type is uniformly dominant, thereby establishing a stronger form of a conjecture of Takahashi.
\end{abstract}
\maketitle
\section{Introduction}
Contravariantly finite subcategories were introduced by Auslander--Smal{\o} \cite{AS1980,AS1981} in connection with their study of preprojective modules and almost split sequences in subcategories.
Resolving subcategories were introduced by Auslander--Bridger \cite{AB1969} in their development of stable module theory and Gorenstein dimension.
Contravariantly finite resolving subcategories are closely related to tilting theory. 
For an artin algebra $\Lambda$ of finite global dimension, a fundamental result of Auslander--Reiten \cite{AR1991} gives a one-to-one correspondence between basic cotilting modules and contravariantly finite resolving subcategories of $\mod\Lambda$.

Let $R$ be a commutative noetherian local ring.
Denote by $\mod R$ the category of finitely generated $R$-modules, and by $\cm(R)$ the full subcategory of $\mod R$ consisting of maximal Cohen--Macaulay $R$-modules.
The Auslander--Buchweitz Cohen--Macaulay approximation theory \cite{AB1989} can be regarded as a commutative-algebraic aspect of tilting theory. 
In this context, Takahashi \cite{T2011} studied contravariantly finite resolving subcategories over commutative noetherian local rings and obtained the following result.
\begin{thm}[Takahashi]\label{int_T2011}
Let $R$ be a henselian Gorenstein local ring, and $\X$ a contravariantly finite resolving subcategory of $\mod R$. Then $\X=\add R$, or $\X=\cm(R)$, or $\X=\mod R$.
\end{thm}
\noindent
Here, $\add R$ stands for the subcategory of $\mod R$ consisting of free modules.
Theorem \ref{int_T2011} indicates that, even without the assumption of finite global dimension, contravariantly finite resolving subcategories are highly restricted in the commutative local setting.
The main result of this paper is the following.
\begin{thm}[Corollary \ref{mainthm}]\label{int_thm1}
Let $R$ be a henselian local ring and $\X$ a contravariantly finite resolving subcategory of $\mod R$. 
Then $\X=\add R$, or $\X=\mod R$, or $R$ is Cohen--Macaulay and $\X=\cm(R)$. 
\end{thm}
\noindent
Thus, this theorem removes the Gorenstein assumption from Theorem \ref{int_T2011} and determines all possible contravariantly finite resolving subcategories over henselian local rings.

We next apply Theorem \ref{int_thm1} to resolving subcategories of finite type. 
Recall that a subcategory $\X$ of $\mod R$ is said to be of finite type if $\add\X$ contains only finitely many isomorphism classes of indecomposable modules, and that a Cohen--Macaulay local ring $R$ is said to have finite CM type if $\cm(R)$ is of finite type.
\begin{thm}[Theorems \ref{thm_fintype} and \ref{thm_conj}]\label{int_thm2}
The following hold.
\begin{enumerate}[\rm(1)]
\item
Let $R$ be a local ring.
If $\X$ is a resolving subcategory of $\mod R$ that contains a nonfree $R$-module and is of finite type, then $R$ is Cohen--Macaulay of finite CM type and $\X=\cm(R)$.
\item
Every Cohen--Macaulay local ring of finite CM type is uniformly dominant.
\end{enumerate}
\end{thm}
\noindent
Since every uniformly dominant local ring is dominant, Theorem \ref{int_thm2}(2) gives a stronger affirmative answer to \cite[Conjecture 9.1]{T2023}. 

The following result is used in the proof of Theorem \ref{int_thm1} and shows that the assumption in \cite[Theorem 3.4]{T2011} that $\Ext_R^i(G,R)=0$ for all $i\gg0$ is superfluous.
\begin{thm}[Theorem \ref{inj}]\label{int_thm3}
Let $R$ be a local ring, and let $\X$ be a resolving subcategory of $\mod R$. 
Suppose that $\X$ contains an $R$-module of infinite projective dimension and that the residue field $k$ has a minimal right $\X$-approximation. 
If $M\in\mod R$ satisfies $\Ext_R^i(X,M)=0$ for all $i\gg0$ and every $X\in\X$, then $\id M<\infty$.
\end{thm}
\noindent
The proof of Theorem \ref{int_thm3} is much simpler than that of \cite[Theorem 3.4]{T2011}.
We also give further applications to test modules for finite projective and injective dimensions; see Propositions \ref{prop_EP} and \ref{prop_EI}.

The organization of this paper is as follows. 
In Section 2, we prove the key lemma and use it to establish Theorem \ref{inj} and Corollary \ref{mainthm}. 
In Section 3, we determine the resolving subcategories of finite type and apply this result to Cohen--Macaulay local rings of finite CM type, giving an affirmative answer to a conjecture of Takahashi.
In Section 4, we give further applications of our results.
\begin{conv}
Throughout this paper, all subcategories are assumed to be strictly full. Unless otherwise specified, $(R,\m,k)$ denotes a commutative noetherian local ring with maximal ideal $\m$ and the residue field $k$. 
\end{conv}
\section{Proof of the main theorem}
The aim of this section is to prove the main theorem of this paper. 
We begin by recalling the notions of contravariantly finite and resolving subcategories.
We refer to \cite{AB1969,AR1991,AS1980,AS1981,T2011} for basic properties of approximations and resolving subcategories.
\begin{dfn}\label{def_appr}
Let $\X$ be a subcategory of $\mod R$ and $M\in\mod R$.
\begin{enumerate}[\rm(1)]
\item 
The {\em additive closure} of $\X$, denoted by $\add\X$, is the subcategory of $\mod R$ consisting of all direct summands of finite direct sums of modules in $\X$.
\item 
A homomorphism $f:X\to M$ with $X\in\X$ is called a {\em right $\X$-approximation} of $M$ if $\hom_R(X',f):\hom_R(X',X)\to\hom_R(X',M)$ is surjective for every $X'\in\X$. 
It is called {\em minimal} if every endomorphism $g:X\to X$ satisfying $f=fg$ is an automorphism. 
The subcategory $\X$ is called {\em contravariantly finite} in $\mod R$ if every module in $\mod R$ has a right $\X$-approximation.
\item 
Let $\cdots\to P_1\to P_0\to M\to0$ be a minimal free resolution of $M$. 
For $n\ge1$, the $n$-th syzygy $\syz^nM$ of $M$ is the image of $P_n\to P_{n-1}$, and we set $\syz^0M=M$.
\item 
The subcategory $\X$ is called {\em resolving} if it contains all finitely generated projective $R$-modules and is closed under direct summands, extensions, and kernels of epimorphisms.
The \emph{resolving closure} of $\X$, denoted by $\res\X$, is the smallest resolving subcategory of $\mod R$ containing $\X$.
\item 
We set ${}^\perp\X=\{N\in\mod R\mid\Ext_R^i(N,X)=0\text{ for all }X\in\X\text{ and }i>0\}$, $\X^\perp=\{N\in\mod R\mid\Ext_R^i(X,N)=0\text{ for all }X\in\X\text{ and }i>0\}$, ${}^{\perp_{\gg0}}\X=\{N\in\mod R\mid\Ext_R^i(N,X)=0\text{ for all }X\in\X\text{ and }i\gg0\}$, and $\X^{\perp_{\gg0}}=\{N\in\mod R\mid\Ext_R^i(X,N)=0\text{ for all }X\in\X\text{ and }i\gg0\}$.
\item
Let $R$ be a Cohen--Macaulay local ring with a canonical module. 
An exact sequence $0\to Y_{M}\to X_{M}\to M\to0$ is called a \emph{maximal Cohen--Macaulay approximation} of $M$ if $X_{M}\in\cm(R)$ and $\id Y_{M}<\infty$. 
Dually, an exact sequence $0\to M\to Y^{M}\to X^{M}\to0$ is called an \emph{FID hull} of $M$ if $\id Y^{M}<\infty$ and $X^{M}\in\cm(R)$; see \cite[Definitions 11.8 and 11.10]{LW}. 
Every finitely generated $R$-module admits a maximal Cohen--Macaulay approximation and an FID hull; see \cite[Theorem 11.17]{LW}.
\end{enumerate}
\end{dfn}
\begin{rem}\label{rmk_appr}
Let $\X$ be a subcategory of $\mod R$.
\begin{enumerate}[\rm(1)]
\item 
The category $\add R$ coincides with the category of finitely generated projective $R$-modules.
\item 
Every right $\X$-approximation is a right $\add\X$-approximation.
\item 
For every $G\in\mod R$, the subcategory $\add G$ is contravariantly finite in $\mod R$.
\item 
Suppose that $\X$ contains $R$ and is closed under direct summands and extensions. 
Then $\X$ is resolving if and only if it is closed under syzygies.
\item
The subcategories ${}^\perp\X$ and ${}^{\perp_{\gg0}}\X$ are resolving.
\item 
Suppose that $\add R\subseteq\X$. 
Then every right $\X$-approximation $X\to M$ is surjective.
\item 
Suppose that $R$ is henselian and that $\X$ is closed under direct summands. 
If $M$ has a right $\X$-approximation, then it has a minimal right $\X$-approximation; see \cite[Corollary 2.5]{T2005}.
\item 
Suppose that $\X$ is closed under extensions and that $f:X\to M$ is a minimal right $\X$-approximation. 
Then $\Ext_{R}^{1}(\X,\ker f)=0$ by Wakamatsu's lemma; see \cite[Lemma 2.1.1]{Xu}.
\item 
Suppose that $\X$ is resolving and that $f:X\to M$ is a minimal right $\X$-approximation. 
Then there exists an exact sequence $0\to Y\to X\xrightarrow{f}M\to0$ with $Y\in\X^\perp$.
\end{enumerate}
\end{rem}
To prove Theorem \ref{inj}, we first establish the following lemma, which plays a central role in the proof.
For a complex $C=(\cdots\to C^{i-1}\xrightarrow{d_C^{i-1}}C^i\xrightarrow{d_C^i}C^{i+1}\to\cdots)$ of $R$-modules and $i\in\mathbb{Z}$, we set $Z^{i}C=\ker d_{C}^{i}$.
\begin{lem}\label{keylem}
Let $A,B\in\mod R$.
Assume that one of the following conditions holds:
\begin{enumerate}[\rm(1)]
\item 
There exists an exact sequence $0\to Y\to X\to k^{\oplus r}\to0$ in $\mod R$ for some integer $r>0$.
\item 
There exists an exact sequence $0\to k^{\oplus r}\to Y\to X\to0$ in $\mod R$ for some integer $r>0$.
\end{enumerate}
If one has $\Ext_R^{\gg0}(A,Y)=\Ext_R^{\gg0}(X,B)=0$, then we have $\pd A<\infty$ or $\id B<\infty$.
\end{lem}
\begin{proof}
(1) Assume that $\pd A=\infty$. 
We show that $\id B<\infty$. 
Let $P$ be a minimal free resolution of $A$.
Set $\mathbf{Y}=\hom_R(P,Y)$, $\mathbf{X}=\hom_R(P,X)$, and $K=\hom_R(P,k^{\oplus r})$.
Applying $\hom_R(P,-)$ to the given exact sequence yields an exact sequence $0\to \mathbf{Y}\xrightarrow{f}\mathbf{X}\xrightarrow{g}K\to0$ of $R$-complexes. 
Since $\H^{i+1}(\mathbf{Y})=\Ext_R^{i+1}(A,Y)=0$ for $i\gg0$, we have an exact sequence $0\to Z^{i}\mathbf{Y}\to \mathbf{Y}^{i}\xrightarrow{b_{i}}Z^{i+1}\mathbf{Y}\to0$ for $i\gg0$.
We show that $b_{i}$ factors through $\mathbf{X}^{i}$ for $i\gg0$.
Since $P$ is minimal and $\m k^{\oplus r}=0$, one has $d_K^i=0$ for every $i$.
Consider the following commutative diagram.
\[
\xymatrix{
0 \ar[r] & \mathbf{Y}^{i} \ar[r]^{f^{i}} \ar[d]^{d_{\mathbf{Y}}^{i}}
& \mathbf{X}^{i} \ar[r]^{g^{i}} \ar[d]^{d_{\mathbf{X}}^{i}} \ar@{-->}[dl]_{\alpha^{i}}
& K^{i} \ar[r] \ar[d]^{d_{K}^{i}=0} & 0\\
0 \ar[r] & \mathbf{Y}^{i+1} \ar[r]^{f^{i+1}} \ar[d]^{d_{\mathbf{Y}}^{i+1}}
& \mathbf{X}^{i+1} \ar[r]^{g^{i+1}} \ar[d]^{d_{\mathbf{X}}^{i+1}}
& K^{i+1} \ar[r] & 0\\
0 \ar[r] & \mathbf{Y}^{i+2} \ar[r]^{f^{i+2}}
& \mathbf{X}^{i+2}. &&
}
\]
As $g^{i+1}d_{\mathbf{X}}^{i}=d_{K}^{i}g^{i}=0$, there exists a homomorphism $\alpha^{i}:\mathbf{X}^{i}\to \mathbf{Y}^{i+1}$ with $f^{i+1}\alpha^{i}=d_{\mathbf{X}}^{i}$. 
A diagram chase, together with the injectivity of $f^{i+1}$ and $f^{i+2}$, shows that $\alpha^{i}f^{i}=d_{\mathbf{Y}}^{i}$ and $d_{\mathbf{Y}}^{i+1}\alpha^{i}=0$.
Hence $\operatorname{Im}(\alpha^{i})\subseteq Z^{i+1}\mathbf{Y}$, and $\alpha^{i}$ induces a homomorphism $\overline{\alpha}^{i}:\mathbf{X}^{i}\to Z^{i+1}\mathbf{Y}$. Thus we have a commutative diagram
\[
\xymatrix{
0 \ar[r] & \mathbf{Y}^{i} \ar[r]^{f^{i}} \ar[d]_{b_{i}}
& \mathbf{X}^{i} \ar[d]^{\alpha^{i}} \ar@{-->}[dl]_{\overline{\alpha}^{i}}\\
0 \ar[r] & Z^{i+1}\mathbf{Y} \ar[r]
& \mathbf{Y}^{i+1}.
}
\]
In particular, $b_{i}=\overline{\alpha}^{i}f^{i}$.
Write $P^{-i}\cong R^{\oplus r_{i}}$. 
Since $\pd A=\infty$, one has $P^{-i}\ne0$, and hence $r_{i}>0$, for every $i\ge0$. 
Thus $\mathbf{X}^{i}\cong X^{\oplus r_{i}}$ and $\mathbf{Y}^{i}\cong Y^{\oplus r_{i}}$. 
As $\Ext_{R}^{\gg0}(X,B)=0$, we have $\Ext_R^{i}(\mathbf{X}^{i},B)=\Ext_R^{i+1}(\mathbf{X}^{i},B)=0$ for $i\gg0$.
Since $b_{i}$ factors through $\mathbf{X}^{i}$ for $i\gg0$, we have $\Ext_{R}^{i}(b_{i},B)=\Ext_{R}^{i+1}(b_{i},B)=0$ for $i\gg0$.
Hence the long exact sequence associated to $0\to Z^{i}\mathbf{Y}\to \mathbf{Y}^{i}\xrightarrow{b_{i}}Z^{i+1}\mathbf{Y}\to0$ yields an exact sequence $0\to\Ext_R^{i}(Y,B)^{\oplus r_{i}}\to D_{i}\xrightarrow{\varepsilon_{i}}D_{i+1}\to0$ for $i\gg0$, where $D_{i}=\Ext_R^{i}(Z^{i}\mathbf{Y},B)$.
Choose $n$ so that the preceding exact sequence exists for every $i\ge n$. 
Note that $D_{n}=\Ext_R^{n}(Z^{n}\mathbf{Y},B)$ is a finitely generated $R$-module. 
Since $D_{n}\xrightarrow{\varepsilon_{n}}D_{n+1}\xrightarrow{\varepsilon_{n+1}}D_{n+2}\to\cdots$ is a sequence of surjective homomorphisms, \cite[Lemma 4.5]{T2011} shows that $\varepsilon_{i}$ is an isomorphism for $i\gg0$. 
It follows that $\Ext_R^{i}(Y,B)^{\oplus r_{i}}=0$ for $i\gg0$. As $r_{i}>0$, we obtain $\Ext_R^{i}(Y,B)=0$ for $i\gg0$.
The long exact sequence associated to $0\to Y\to X\to k^{\oplus r}\to0$, together with $\Ext_{R}^{\gg0}(X,B)=\Ext_{R}^{\gg0}(Y,B)=0$, yields $\Ext_{R}^{\gg0}(k,B)=0$, which implies that $B$ has finite injective dimension.
The assertion (2) is proved in the same way.
\end{proof}
We have now reached the main result of this section. 
The following theorem shows that the assumption $\Ext_{R}^{\gg0}(G,R)=0$ in \cite[Theorem 3.4]{T2011} is unnecessary.
\begin{thm}\label{inj}
Let $\X$ be a resolving subcategory of $\mod R$. 
If there exists an $R$-module $G$ in $\X$ with $\pd G=\infty$ and $k$ has a minimal right $\X$-approximation, then every module in $\X^{\perp_{\gg 0}}$ has finite injective dimension.
\end{thm}
\begin{proof}
There exists an exact sequence $0\to Y\to X\to k\to 0$ with $X\in \X,Y\in \X^\perp$. 
If $M\in \X^{\perp_{\gg0}}$, then $\Ext_R^{>0}(G,Y)=0$ and $\Ext_R^{\gg 0}(X,M)=0$ hold. 
Therefore, it follows from Lemma \ref{keylem} that $M$ has finite injective dimension.
\end{proof}
As an application of Theorem \ref{inj}, we determine the possible forms of contravariantly finite resolving subcategories over henselian local rings. 
The proof uses Theorem \ref{inj} and the arguments in the proofs of \cite[Corollaries 3.9 and 3.13]{T2011}.
\begin{cor}\label{mainthm}
Let $R$ be a henselian local ring, and let $\X$ be a contravariantly finite resolving subcategory of $\mod R$.
Then $\X=\add R$, $\X=\mod R$, or $R$ is Cohen--Macaulay and $\X=\cm(R)$.
\end{cor}
\begin{proof}
Suppose that $\X$ is a proper subcategory of $\mod R$. 
By \cite[Lemma 3.10]{T2011}, the residue field $k$ does not belong to $\X$. 
Hence, there exists an exact sequence $0\to Y\to X\to k\to0$ with $X\in\X$ and $0\neq Y\in\X^{\perp}$.
Suppose first that every module in $\X$ has finite projective dimension.
Then \cite[Proposition 3.14]{T2011} yields $\X=\add R$.
Suppose next that $\X$ contains an $R$-module $G$ of infinite projective dimension. 
Since $Y\in\X^{\perp}$, Theorem \ref{inj} shows that $Y$ has finite injective dimension. 
Hence the Bass conjecture \cite{PS1973,Rob1987} implies that $R$ is Cohen--Macaulay.
We show that $\X=\cm(R)$. 
Let $0\neq X'\in\X$. 
Since $Y\in\X^\perp$ and $\id Y<\infty$, \cite[Lemma 3.12]{T2011} gives $\depth R-\depth X'=\sup\{i\mid\Ext_{R}^{i}(X',Y)\neq0\}\le0$. 
Hence $\depth X'=\dim R$, and therefore $\X\subseteq\cm(R)$.
Conversely, let $M\in\cm(R)$.
 By \cite[Lemma 3.11]{T2011}, it is enough to show that $M\in{}^{\perp}(\X^{\perp})$. 
Let $0\neq N\in\X^{\perp}$. 
Since $G\in\X$, Theorem \ref{inj} shows that $N$ has finite injective dimension. 
Hence \cite[Lemma 3.12]{T2011} gives $\sup\{i\mid\Ext_R^i(M,N)\neq0\}=\depth R-\depth M=0$. 
Thus $M\in{}^{\perp}(\X^{\perp})$, and hence $M\in\X$. 
Therefore $\cm(R)\subseteq\X$, and consequently $\X=\cm(R)$.
\end{proof}
\begin{rem}
Using our result, we obtain an affirmative answer to \cite[Question 40]{T2021a} under mild assumptions.
Let $R$ be a henselian local ring which is a homomorphic image of a Gorenstein ring, and $\X$ a resolving subcategory of $\mod R$ which is neither $\add R$, $\cm(R)$, nor $\mod R$.
If $\mod\X$ is abelian and $\X$ is not contravariantly infinite (see \cite{Tan2026}), then $k$ has a right $\X$-approximation by \cite[Corollary 32(2)]{T2021a}.
The same argument as in the proof of Corollary \ref{mainthm} shows that $\X$ is contained in $\cm(R)$.
It follows from \cite[Theorem 27]{T2021a} that $\X$ is contravariantly finite.
In particular, $\mod\X$ is abelian if and only if $\X$ is closed under kernels if and only if $\X$ has pseudo-kernels; see \cite{T2021a}.
\end{rem}
\section{Resolving subcategories of finite type}
In this section, we establish Theorem \ref{thm_fintype} and apply it to a conjecture of Takahashi on dominant local rings.
A subcategory $\X$ of $\mod R$ is said to be {\em of finite type} if $\add\X$ contains only finitely many isomorphism classes of indecomposable modules. 
If $R$ is Cohen--Macaulay, we say that $R$ has \emph{finite CM type} if $\cm(R)$ is of finite type.
The following facts follow from \cite[Theorem 2.2]{LW} and \cite[Proposition 2.18]{LW}, respectively. 
We use them without further mention in the rest of this paper.
We denote by $\widehat{R}$ the $\m$-adic completion of $R$, and for $M\in\mod R$ we set $\widehat{M}=M\otimes_{R}\widehat{R}$.
\begin{lem}\label{lem_compl}
The following statements hold.
\begin{enumerate}[\rm(1)]
\item 
Let $\X$ be a subcategory of $\mod R$ closed under finite direct sums and direct summands. 
Then $\X$ is of finite type if and only if $\X$ has an additive generator.
\item 
Let $M,N$ be finitely generated $R$-modules. 
Then $M\in\add N$ if and only if $\widehat{M}\in\add_{\widehat{R}}\widehat{N}$.
\end{enumerate}
\end{lem}
Using the technique developed in \cite[Section 4]{CPST}, we reduce the study of resolving subcategories of finite type to the complete case.
\begin{thm}\label{thm_fintype}
Let $\X$ be a resolving subcategory of finite type in $\mod R$. Then either $\X=\add R$, or $R$ is Cohen--Macaulay of finite CM type and $\X=\cm(R)$.
\end{thm}
\begin{proof}
There exists $G\in\X$ such that $\X=\add G$. 
We first show that $\add_{\widehat{R}}\widehat{G}$ is a resolving subcategory of $\mod\widehat{R}$. 
The argument in the proof of \cite[Theorem 4.3]{CPST} shows that $\add_{\widehat R}\widehat G$ is closed under extensions. 
It follows that $\add_{\widehat R}\widehat G$ contains $\widehat{R}$ and is closed under syzygies, since $\X$ is resolving.
Thus $\add_{\widehat{R}}\widehat{G}$ is resolving.
There is nothing to prove if $\X=\add R$. 
Suppose that $\add R\ne\X$.
Then $\add_{\widehat{R}}\widehat{R}\subsetneq\add_{\widehat R}\widehat G$. 
Since $\add_{\widehat{R}}\widehat{G}$ is of finite type, it is contravariantly finite in $\mod\widehat{R}$.
Hence Corollary \ref{mainthm} shows that either $\add_{\widehat{R}}\widehat{G}=\mod\widehat{R}$, or $\widehat{R}$ is Cohen--Macaulay and $\add_{\widehat{R}}\widehat{G}=\cm(\widehat{R})$.
In the former case, one has $\X=\add G=\mod R$. 
Since $\X$ is of finite type, $R$ has finite representation type, and hence $R$ is artinian by \cite[Theorem 3.3]{LW}. 
Thus $R$ is Cohen--Macaulay of finite CM type and $\X=\cm(R)$. 
In the latter case, $R$ is Cohen--Macaulay and $G\in\cm(R)$.
It follows that  $\add G=\cm(R)$, and hence $R$ has finite CM type.
\end{proof}
We next apply Theorem \ref{thm_fintype} to Cohen--Macaulay local rings of finite CM type.
We recall the notions of (uniformly) dominant local rings and the dominant index.
We denote by $\ds(R)$ the singularity category of $R$, that is, the Verdier quotient of $\db(R)$ by the thick subcategory of perfect complexes.
We use the following notation and terminology; see \cite[Definitions 3.1, 3.3, 3.4 and 3.6]{KT2026}.
\begin{dfn}\label{def_dom}
Let $\T$ be a triangulated category.
\begin{enumerate}[\rm(1)]
\item 
For subcategories $\X,\Y$ of $\mod R$, let $\X\circ\Y$ denote the subcategory of $\mod R$ consisting of modules $M$ such that there exists an exact sequence $0\to X\to M\to Y\to0$ with $X\in\X$ and $Y\in\Y$. 
We set ${\lbrack\X\rbrack}_0=0$ and ${\lbrack\X\rbrack}_1=\add\bigl(\{R\}\cup\{\syz^iX\mid X\in\X,\ i\ge0\}\bigr)$. 
For $n>1$, we set ${\lbrack\X\rbrack}_n=\add\bigl({\lbrack\X\rbrack}_{n-1}\circ{\lbrack\X\rbrack}_{1}\bigr)$.
\item 
Let $\X,\Y$ be subcategories of $\T$. 
Let $\X\ast\Y$ denote the subcategory of $\T$ consisting of objects $E$ such that there exists an exact triangle $X\to E\to Y\to X[1]$ with $X\in\X$ and $Y\in\Y$. 
We set ${\langle\X\rangle}_0^\T=0$ and ${\langle\X\rangle}_1^\T=\add\{X[i]\mid X\in\X,\ i\in\ZZ\}$. For $n>1$, we set ${\langle\X\rangle}_n^\T=\add\bigl({\langle\X\rangle}_{n-1}^\T\ast{\langle\X\rangle}_{1}^{\T}\bigr)$.
\item 
The ring $R$ is {\em dominant} if $k\in\thick_{\ds(R)}X$ for every nonzero object $X\in\ds(R)$. 
The {\em dominant index} of $R$ is defined as $\operatorname{dx}(R)=\inf\{n\in\ZZ_{\ge-1}\mid k\in{\langle X\rangle}_{n+1}^{\ds(R)}\text{ for every }0\ne X\in\ds(R)\}$. 
The ring $R$ is {\em uniformly dominant} if $\operatorname{dx}(R)<\infty$.
\end{enumerate}
\end{dfn}
For a singular local ring $R$, the {\em (resolving) core} $\core R$ is defined to be the intersection of all resolving subcategories of $\mod R$ containing a module of infinite projective dimension; see \cite[Definition 3.1(3)]{T2021b}.
Takahashi conjectured in \cite[Conjecture 4.2]{T2021b} that every singular Cohen--Macaulay local ring of finite CM type satisfies $\core R=\cm(R)$. 
This conjecture was later reformulated as \cite[Conjecture 9.1]{T2023}, which asserts that every Cohen--Macaulay local ring of finite CM type is dominant; see \cite[Proposition 6.2(2)]{T2023}. 
Several partial affirmative results have been obtained; see \cite[Proposition 6.6]{T2023}, \cite[Proposition 6.10]{KT2026}, and \cite[Theorem 1.2]{Liu2026} for instance. 
As an application of Theorem \ref{thm_fintype}, we obtain the following stronger result.
\begin{thm}\label{thm_conj}
Every Cohen--Macaulay local ring of finite CM type is uniformly dominant.
\end{thm}
\begin{proof}
Set $d=\dim R$. 
We may assume that $R$ is singular. 
Since $R$ is of finite CM type, let $G_1,\dots,G_s$ be representatives of the isomorphism classes of nonfree indecomposable maximal Cohen--Macaulay $R$-modules, and put $G=R\oplus G_1\oplus\cdots\oplus G_s$. 
Then $\cm(R)=\add G$. 
Since $\res G_i$ is a resolving subcategory of finite type containing the nonfree module $G_i$, Theorem \ref{thm_fintype} yields $\res G_i=\cm(R)$.
It follows that $\bigcup_{m\ge0}{\lbrack G_i\rbrack}_{m}=\cm(R)$. 
Hence there exists $n_i\ge1$ such that $\syz^d k\in{\lbrack G_i\rbrack}_{n_i}$. 
Put $n=\max\{n_1,\dots,n_s\}$. 
Then $\syz^d k\in\lbrack G_i\rbrack_n$ for every $1\le i\le s$. 
Let $M$ be a finitely generated $R$-module with $\pd M=\infty$. 
Then $\syz^dM$ has a nonfree indecomposable direct summand, which is maximal Cohen--Macaulay and hence isomorphic to $G_i$ for some $i$. 
Therefore $\lbrack G_i\rbrack_n\subseteq\lbrack\syz^dM\rbrack_n\subseteq\lbrack M\rbrack_n$, and hence $\syz^d k\in\lbrack M\rbrack_n$. 
It follows from \cite[Proposition 3.8]{KT2026} that $\operatorname{dx}(R)\le n-1<\infty$. 
Thus $R$ is uniformly dominant.
\end{proof}
\begin{rem}
Alternatively, the proof of Theorem \ref{thm_conj} can be reduced to the case where $R$ is complete.
Indeed, by \cite[Theorem 1.2]{DKLO}, $R$ has finite CM type if and only if $\widehat{R}$ has finite CM type, while \cite[Corollary 6.3]{T2026} shows that $R$ is uniformly dominant if and only if $\widehat{R}$ is uniformly dominant; see also the proof of \cite[Proposition 4.4(2)]{Liu2026}.
\end{rem}
Combining Theorem \ref{thm_conj} with \cite[Proposition 6.2(2)]{T2023}, we obtain the following.
\begin{cor}
Conjecture \cite[Conjecture 9.1]{T2023} holds true, and hence so does \cite[Conjecture 4.2]{T2021b}.
\end{cor}
\section{Further applications}
In this section, we give further applications of the results obtained in Sections 2 and 3.
The following corollary extends \cite[Corollary 3.16]{T2011} from Gorenstein local rings to Cohen--Macaulay local rings with canonical modules. 
It also generalizes \cite[Theorem 6.5]{T2013} by removing the assumptions that $\X$ is contained in $\cm(R)$ and that $\X$ contains a canonical module. 
\begin{cor}\label{cor_cm}
Let $R$ be a henselian Cohen--Macaulay local ring with a canonical module.
Let $\X$ be a resolving subcategory of $\mod R$.
Then $\X$ is contravariantly finite in $\mod R$ if and only if $\X=\add R$, $\X=\cm(R)$, or $\X=\mod R$.
\end{cor}
\begin{proof}
The forward implication follows from Corollary \ref{mainthm}, while the converse follows from the theory of maximal Cohen--Macaulay approximations \cite{AB1989}.
\end{proof}
\begin{rem} 
Let $R$ be a henselian Cohen--Macaulay local ring with a canonical module $\omega$. 
Combining Corollary \ref{cor_cm} with the canonical duality $\hom_R(-,\omega)\colon \cm(R)^{\mathrm{op}}\xrightarrow{\sim}\cm(R)$, we see that the covariantly finite coresolving subcategories of $\cm(R)$ are precisely $\add\omega$ and $\cm(R)$. 
Thus Corollary \ref{cor_cm} yields a higher-dimensional commutative analogue of \cite[Corollary 0.3]{KS2003}. 
\end{rem} 
For a semidualizing $R$-module $C$, let $\G_C(R)$ denote the category of totally $C$-reflexive $R$-modules; see \cite{ATY2005} for the definitions and basic properties.
Theorem \ref{inj} yields the following semidualizing version of \cite[Theorem C]{CPST}.
\begin{cor}\label{cor_gc}
Let $C$ be a semidualizing $R$-module. 
If $k$ has a minimal right $\G_C(R)$-approximation, then either $\G_C(R)=\add R$, or $R$ is Cohen--Macaulay and $C$ is a canonical module of $R$.
\end{cor}
\begin{proof}
Assume that $\G_C(R)\ne\add R$, and choose a nonfree module $G\in\G_C(R)$. 
Then $\pd G=\infty$. 
Since $\G_C(R)$ is resolving and $C\in\G_C(R)^\perp$, Theorem \ref{inj} yields $\id C<\infty$.
Hence $C$ is a dualizing module, so $R$ is Cohen--Macaulay and $C$ is a canonical module of $R$.
\end{proof}
Theorem \ref{thm_fintype} also yields the following semidualizing version of \cite[Theorem B]{CPST}.
\begin{cor}\label{cor_gc_fintype}
Let $C$ be a semidualizing $R$-module. Then $\G_C(R)$ is of finite type if and only if either $\G_C(R)=\add R$, or $R$ is Cohen--Macaulay of finite CM type and $C$ is a canonical module of $R$.
\end{cor}
The following result gives restrictions on a resolving subcategory $\X$ when the residue field admits a right $\X$-approximation, without assuming that $\X$ is contravariantly finite.
When $R$ is Cohen--Macaulay, we denote by $\cm_0(R)$ the subcategory of $\cm(R)$ consisting of modules that are locally free on the punctured spectrum.
\begin{prop}\label{sensei_prop}
Let $R$ be a henselian local ring and $\X$ a resolving subcategory of $\mod R$. 
If $k$ has a right $\X$-approximation, then one of the following conditions holds.
\begin{enumerate}[\rm(1)]
\item 
Either $k\in\X$ or $\X=\add R$.
\item 
The ring $R$ is Cohen--Macaulay and $\X$ is contained in $\cm (R)$. If, moreover, $R$ admits a canonical module $\omega$ and $\omega\in\X$, then one has $\cm_0(R)\subseteq\X$.
\end{enumerate}
\end{prop}
\begin{proof}
We set $d=\dim R$ and assume that $k\notin\X$. 
There exists an exact sequence $(\star)$ $0\to Y\to X\to k\to0$ with $X\in\X$ and $0\neq Y\in\X^\perp$. 
If $\X=\add R$, then assertion (1) holds, so assume that $\X\neq\add R$. 
If every module in $\X$ had finite projective dimension, then \cite[Lemma 3.12]{T2011} would give $\pd X'=\sup\{i\mid\Ext_R^i(X',Y)\neq0\}=0$ for every nonzero $X'\in\X$, and hence $\X=\add R$, contrary to our assumption.
Thus $\X$ contains a module of infinite projective dimension, and Theorem \ref{inj} yields $\id Y<\infty$. 
Hence the Bass conjecture \cite{PS1973,Rob1987} implies that $R$ is Cohen--Macaulay, while \cite[Lemma 3.12]{T2011} gives $\depth R-\depth X'=\sup\{i\mid\Ext_R^i(X',Y)\neq0\}=0$ for every nonzero $X'\in\X$. 
Therefore $\X\subseteq\cm(R)$. 
Assume now that $R$ admits a canonical module $\omega$ and that $\omega\in\X$. 
If $d=0$, then $\id Y=0$, so $(\star)$ splits and $k\in\X$, a contradiction. 
Thus $d>0$. 
The module $Y$ has finite $\add\omega$-resolution dimension at most $d-1$; see \cite[Exercise 3.3.28]{BH}. 
Together with $(\star)$, an $\add\omega$-resolution of $Y$ of length at most $d-1$ gives an $\X$-resolution of $k$ of length at most $d$, and hence $\Omega^dk\in\X$; see \cite{AB1969}. 
It follows from \cite[Corollary 2.6]{T2010} that $\cm_0(R)=\res\Omega^dk\subseteq\X$.
\end{proof}
\begin{rem}
Let $R$ be a henselian Cohen--Macaulay local ring with a canonical module. 
Suppose that $R$ is locally Gorenstein on the punctured spectrum but does not have an isolated singularity. 
Then $k$ has a right $\cm_0(R)$-approximation, although $\cm_0(R)$ is not contravariantly finite in $\mod R$. 
Indeed, let $0\to Y\to X\to k\to0$ be a maximal Cohen--Macaulay approximation of $k$. 
For every $\p\ne\m$, the module $X_\p\cong Y_\p$ is maximal Cohen--Macaulay and has finite injective dimension over the Gorenstein local ring $R_\p$, and hence is free. 
Thus $X\in\cm_0(R)$, so $X\to k$ is a right $\cm_0(R)$-approximation.
On the other hand, $\add R\subsetneq\cm_0(R)\subsetneq\cm(R)$, and hence Corollary \ref{mainthm} shows that $\cm_0(R)$ is not contravariantly finite in $\mod R$.
\end{rem}
We next apply Lemma \ref{keylem} to modules that detect finite projective dimension via the vanishing of Ext.
Following \cite{CDT2014}, set $\operatorname{EP}(R)=\{M\in\mod R\mid \Ext_R^{\gg0}(X,M)=0\text{ implies }\pd_RX<\infty\text{ for every }X\in\mod R\}$.
\begin{prop}\label{prop_EP}
Let $R$ be a non-Gorenstein local ring, and $M$ a finitely generated $R$-module. 
If there exists an exact sequence $0\to k\to M\to G\to0$ such that $\Ext_R^{\gg0}(G,R)=0$, then $M\in\operatorname{EP}(R)$.
\end{prop}
We provide an application of Proposition \ref{prop_EP}.
\begin{ex}
Let $R$ be a non-Gorenstein local ring, and let $G$ be a nonfree totally reflexive $R$-module.
Since $\Ext_R^1(G,k)\ne0$, there exists a nonsplit exact sequence $\xi:0\to k\to M_{G}\to G\to0$. 
As $\Ext_R^{>0}(G,R)=0$, Proposition \ref{prop_EP} yields $M_{G}\in\operatorname{EP}(R)$. 
Moreover, $k$ is not a direct summand of $M_{G}$.
Indeed, suppose that $k$ is a direct summand of $M_G$, and consider the composite $k\to M_G\to k$ of the injection in $\xi$ with the projection onto this direct summand. 
This map is either zero or an isomorphism. 
If it is zero, then the projection $M_G\to k$ factors through the cokernel $G$.
This yields a split epimorphism $G\to k$. 
Hence $k$ is totally reflexive, which forces $R$ to be Gorenstein.
This is a contradiction.
If it is an isomorphism, then the injection $k\to M_G$ is a split monomorphism, and hence $\xi$ splits, again a contradiction.
If, in addition, $\depth R>0$, then $M_{G}^{**}\cong G$.
This implies that pairwise nonisomorphic nonfree totally reflexive modules give rise in this way to pairwise nonisomorphic modules in $\operatorname{EP}(R)$ that have no direct summand isomorphic to $k$.
\end{ex}
\begin{rem}\label{rem_mM}
Let $R$ be a singular local ring, and let $0\ne M\in\mod R$ with $\pd M<\infty$. 
Then one has $\m M\in\operatorname{EP}(R)$ by Lemma \ref{keylem}(1). 
However, this also follows from the Ext analogue of \cite[Theorem 3.3]{IP2007}.
\end{rem}
We finally consider the corresponding test modules for finite injective dimension.
We set $\operatorname{EI}(R)=\{M\in\mod R\mid \Ext_R^{\gg0}(M,N)=0\text{ implies }\id N<\infty\text{ for every }N\in\mod R\}$.
\begin{prop}\label{prop_EI}
Assume that $R$ is a singular Cohen--Macaulay local ring with a canonical module. 
The following statements hold.
\begin{enumerate}[\rm(1)]
\item 
If $0\to Y_k\to X_k\to k\to0$ is a maximal Cohen--Macaulay approximation of $k$, then $X_k\in\operatorname{EI}(R)$.
\item 
If $0\to k\to Y^k\to X^k\to0$ is an FID hull of $k$, then $X^k\in\operatorname{EI}(R)$.
\end{enumerate}
\end{prop}
\begin{ac}
The authors would like to thank their supervisor Ryo Takahashi for valuable comments and helpful discussions, which have further polished up the draft of this paper. 
They also thank Toshinori Kobayashi for pointing out the connection between \cite[Theorem 3.3]{IP2007} and Remark \ref{rem_mM}.
The authors used OpenAI's ChatGPT to assist in organizing the argument of \cite[Theorem 3.4]{T2011} and in proofreading the manuscript. 
The authors independently checked all arguments and are fully responsible for the content of this paper. 
Mifune was partly supported by Grant-in-Aid for JSPS Fellows 25KJ1386.
\end{ac}


\begin{thebibliography}{99}
\bibitem{ATY2005}
{\sc T. Araya; R. Takahashi; Y. Yoshino}, Homological invariants associated to semi-dualizing bimodules, {\em J. Math. Kyoto Univ.} {\bf 45} (2005), no. 2, 287--306.
\bibitem{AB1969}
{\sc M. Auslander; M. Bridger}, Stable module theory, Memoirs of the American Mathematical Society {\bf 94}, {\it American Mathematical Society, Providence, RI}, 1969.
\bibitem{AB1989}
{\sc M. Auslander; R.-O. Buchweitz}, The homological theory of maximal Cohen--Macaulay approximations, {\em M\'em. Soc. Math. France (N.S.)} No. {\bf 38} (1989), 5--37.
\bibitem{AR1991}
{\sc M. Auslander; I. Reiten}, Applications of contravariantly finite subcategories, {\em Adv. Math.} {\bf 86} (1991), no. 1, 111--152.
\bibitem{AS1980}
{\sc M. Auslander; S. O. Smal{\o}}, Preprojective modules over Artin algebras, {\em J. Algebra} {\bf 66} (1980), no. 1, 61--122.
\bibitem{AS1981}
{\sc M. Auslander; S. O. Smal{\o}}, Almost split sequences in subcategories, {\em J. Algebra} {\bf 69} (1981), no. 2, 426--454.
\bibitem{BH}
{\sc W. Bruns; J. Herzog}, Cohen--Macaulay rings, revised edition, Cambridge Studies in Advanced Mathematics {\bf 39}, {\it Cambridge University Press, Cambridge}, 1998.
\bibitem{CDT2014}
{\sc O. Celikbas; H. Dao; R. Takahashi}, Modules that detect finite homological dimensions, {\em Kyoto J. Math.} {\bf 54} (2014), no. 2, 295--310.
\bibitem{CPST}
{\sc L. W. Christensen; G. Piepmeyer; J. Striuli; R. Takahashi}, Finite Gorenstein representation type implies simple singularity, {\em Adv. Math.} {\bf 218} (2008), no. 4, 1012--1026.
\bibitem{DKLO}
{\sc S. Dey; K. Kimura; J. Liu; Y. Otake}, On local rings of finite syzygy representation type, preprint (2025), {\tt arXiv:2507.17097v2}.
\bibitem{IP2007}
{\sc S. B. Iyengar; T. J. Puthenpurakal}, Hilbert--Samuel functions of modules over Cohen--Macaulay rings, {\em Proc. Amer. Math. Soc.} {\bf 135} (2007), no. 3, 637--648.
\bibitem{KT2026}
{\sc T. Kobayashi; R. Takahashi}, On the ubiquity of uniformly dominant local rings, preprint (2026), {\tt arXiv:2603.10810v1}.
\bibitem{KS2003}
{\sc H. Krause; \O. Solberg}, Applications of cotorsion pairs, {\em J. London Math. Soc. (2)} {\bf 68} (2003), no. 3, 631--650.
\bibitem{LW}
{\sc G. J. Leuschke; R. Wiegand}, Cohen--Macaulay representations, Math. Surveys Monogr. {\bf 181}, {\em American Mathematical Society, Providence, RI}, 2012.
\bibitem{Liu2026}
{\sc J. Liu}, On Takahashi's questions about dominant local rings, preprint (2026), {\tt arXiv:2608.14283v1}.
\bibitem{PS1973}
{\sc C. Peskine; L. Szpiro}, Dimension projective finie et cohomologie locale. Applications \`a la d\'emonstration de conjectures de M. Auslander, H. Bass et A. Grothendieck, {\em Publ. Math. Inst. Hautes \'Etudes Sci.} {\bf 42} (1973), 47--119.
\bibitem{Rob1987}
{\sc P. Roberts}, Le th\'eor\`eme d'intersection, {\em C. R. Acad. Sci. Paris S\'er. I Math.} {\bf 304} (1987), no. 7, 177--180.
\bibitem{T2005}
{\sc R. Takahashi}, On the category of modules of Gorenstein dimension zero, {\em Math. Z.} {\bf 251} (2005), no. 2, 249--256.
\bibitem{T2010}
{\sc R. Takahashi}, Classifying thick subcategories of the stable category of Cohen--Macaulay modules, {\em Adv. Math.} {\bf 225} (2010), no. 4, 2076--2116.
\bibitem{T2011}
{\sc R. Takahashi}, Contravariantly finite resolving subcategories over commutative rings, {\em Amer. J. Math.} {\bf 133} (2011), no. 2, 417--436.
\bibitem{T2013}
{\sc R. Takahashi}, Classifying resolving subcategories over a Cohen--Macaulay local ring, {\em Math. Z.} {\bf 273} (2013), no. 1, 569--587.
\bibitem{T2021b}
{\sc R. Takahashi}, Intersections of resolving subcategories and intersections of thick subcategories, {\em Eur. J. Math.} {\bf 7} (2021), no. 4, 1767--1790.
\bibitem{T2021a}
{\sc R. Takahashi}, Resolving subcategories whose finitely presented module categories are abelian, {\em C. R. Math. Acad. Sci. Paris} {\bf 359} (2021), no. 5, 577--592.
\bibitem{T2023}
{\sc R. Takahashi}, Dominant local rings and subcategory classification, {\em Int. Math. Res. Not. IMRN} (2023), no. 9, 7259--7318.
\bibitem{T2026}
{\sc R. Takahashi}, Uniformly dominant local rings and Orlov spectra of singularity categories, {\em Math. Z.} {\bf 312} (2026), no. 4, Paper No. 117, 27 pp.
\bibitem{Tan2026}
{\sc G. Tanigawa}, Contravariantly infinite resolving subcategories, preprint (2026), {\tt arXiv:2603.07945v1}.
\bibitem{Xu}
{\sc J. Xu}, Flat covers of modules, Lecture Notes in Mathematics {\bf 1634}, {\it Springer-Verlag, Berlin}, 1996.
\end{thebibliography}
\end{document}